\documentclass[10pt, a4paper]{amsart}
\usepackage{latexsym, amsfonts, euscript}

\def\Z{{\mathbb Z}}

\def\PP{{\mathbb P}}

\def\Fq{{\mathbb F}_q}
\def\D{{\mathcal D}}
\def\E{{\mathcal E}}
\def\hodge{{\mathfrak h}}

\def\supp{{\rm supp}}
\def\rank{{\rm rank}}

\def\a{{\alpha}}

\def\d{{\delta}}
\def\e{\mathbf{e}}
\def\bfu{\mathbf{u}}
\def\bfv{\mathbf{v}}

\newcommand\vspan{\mathop{\rm span}}
\newtheorem{theorem}{Theorem}
\newtheorem{proposition}[theorem]{Proposition}
\newtheorem{lemma}[theorem]{Lemma}
\newtheorem{corollary}[theorem]{Corollary}

\newtheorem*{theorem*}{Theorem}
\newtheorem*{conjecture*}{Conjecture}
\theoremstyle{remark}
\newtheorem{remark}[theorem]{Remark}

\begin{document}

\title{Decomposable Subspaces,  Linear Sections of Grassmann Varieties, 
and Higher Weights of Grassmann Codes} 

\author{Sudhir R. Ghorpade}
\address{Department of Mathematics, 
Indian Institute of Technology Bombay,\newline \indent
Powai, Mumbai 400076, India.}
\email{srg@math.iitb.ac.in}

\author{Arunkumar R. Patil}
\address{Department of Mathematics,
Indian Institute of Technology Bombay,\newline \indent
Powai, Mumbai 400076, India \newline \indent
and \newline \indent
Shri Guru Gobind Shinghji Institute  of Engineering \& Technology,\newline \indent
Vishnupuri, Nanded 431 606, India}
\email{arun.iitb@gmail.com}

\author{Harish K. Pillai}
\address{Department of Electrical Engineering,
Indian Institute of Technology Bombay,\newline \indent
Powai, Mumbai 400076, India.}
\email{hp@ee.iitb.ac.in}

\date{October 25, 2007}


 
\begin{abstract}
Given a homogeneous component of an exterior algebra, we characterize those subspaces in which every nonzero element is decomposable. In geometric terms, this corresponds to  characterizing the projective linear subvarieties of the Grassmann variety with its Pl\"ucker embedding. When the base field is finite, we consider the more general question of determining the maximum number of points on sections of Grassmannians by linear subvarieties of a fixed (co)dimension.  This corresponds to a known open problem of determining the complete weight hierarchy of 
linear error correcting codes associated to Grassmann varieties.  
We recover most of the known results as well as prove some new results. In the process we obtain, and utilize, a simple generalization of the Griesmer-Wei bound for arbitrary linear codes.  
\end{abstract}

\maketitle

\section{Introduction}
\label{sec:intro}

Let $V$ be an $m$-dimensional vector space over a field $F$. Given a positive integer $\ell$ with $\ell \le m$, consider the $\ell^{\rm th}$ exterior power $\bigwedge^{\ell}V$ of $V$. A nonzero element $\omega\in \bigwedge^{\ell}V$ is said to be {\em   decomposable} if $\omega = v_1\wedge \cdots \wedge v_{\ell}$ for some $v_1, \dots , v_{\ell}\in V$. A subspace of $\bigwedge^{\ell}V$ is {\em   decomposable} if all of its nonzero elements are decomposable. 
In the first part of this paper, we consider the following question: what are all possible   decomposable subspaces of $\bigwedge^{\ell}V$, and, in particular, what is the maximum possible dimension of a decomposable subspace of $\bigwedge^{\ell}V$? We answer this by proving a characterization of decomposable 
subspaces of $\bigwedge^{\ell}V$. This result can be viewed as an algebraic counterpart 
of the combinatorial structure theorem for the so called closed families of subsets of 
a finite set (cf. \cite[Thm. 4.2]{GL2}). As a corollary, we obtain that the maximum possible dimension of a   decomposable subspace of $\bigwedge^{\ell}V$ is $\max\{\ell,m-\ell\} + 1$. 
In geometric terms, this corresponds to characterizing the projective linear subvarieties (with respect to the Pl\"ucker embedding) of the Grassmann variety $G_{\ell,m}$ of all $\ell$-dimensional subspaces of $V$, and showing that the maximum possible (projective) dimension of such a linear subvariety is $\max\{\ell,m-\ell\}$. Briefly speaking, the characterization of decomposable subspaces states that they are necessarily one among the two types of subspaces that are described explicitly. Subsequently, using the Hodge star operator, we observe that a nice duality prevails among the two types of decomposable subspaces. 

In the second part of this paper, we consider the case when $F$ is the finite field $\Fq$ 
with $q$ elements. For a fixed nonnegative integer $s$, we consider the linear sections
$L\cap G_{\ell,m}$ of the Grassmann variety $G_{\ell,m}$ (with its canonical Pl\"ucker embedding) by a linear subvariety $L$ of $\PP(\bigwedge^{\ell}V)$ of dimension $s$, and we ask what is the maximum number of $\Fq$-rational points that such a linear section can have. In light of the abovementioned corollary of the characterization of decomposable subspaces, it is evident that when $s$ is small, or more precisely, when $s\le \max\{\ell,m-\ell\}$, the maximum number is $1+q+q^2+\cdots + q^s$. But for a general $s$, the answer does not seem to be known. In fact, enumerative as well as geometric aspects of linear sections of $G_{\ell,m}$ are not particularly well-understood, in general, except in special cases such as those when the linear sections 
are Schubert subvarieties of $G_{\ell,m}$. (See, for example, Section 6 of \cite{GL2} and the references therein.) However, the above question admits an equivalent formulation in terms of linear error correcting codes, and as such, it has been considered by various authors. Indeed, if we let $C(\ell,m)$ denote the linear code associated to $G_{\ell,m}(\Fq)\hookrightarrow \PP(\bigwedge^{\ell}V)$, 
then its $r^{\rm th}$ higher weight (see Section \ref{sec:Griesmer} for definitions) 
is given by
$$
d_r(C(\ell,m)) = n - \max_L |L\cap G_{\ell,m} (\Fq)|  
$$
where the maximum is taken over projective linear subspaces $L$ of $\PP(\bigwedge^{\ell}\Fq^m)$ of codimension $r$, and where $n$ denotes the Gaussian binomial coefficient defined by
$$
 n=|G_{\ell,m}(\Fq)|= {{m}\brack{\ell}}_q :=\frac{(q^m-1)(q^m-q)\cdots(q^m-q^{\ell-1})}{(q^{\ell}-1)(q^{\ell}-q)\cdots
(q^{\ell}-q^{\ell-1})}.
$$
With this in view, we shall now consider 
the equivalent question of determining 
$d_r = d_r(C(\ell,m))$ for any $r\ge 0$, where 
$d_0:=0$, by convention. This question is open, in general, and the known results can be summarized as follows. From general facts in Coding Theory and the fact that the embedding $G_{\ell,m}(\Fq) \hookrightarrow \PP(\bigwedge^{\ell}\Fq^m)$ is nondegenerate, one knows that 
$$
0= d_0 < d_1<d_2<\cdots<d_k=n \quad \text{where} \quad k := {{m}\choose{\ell}},
$$
and also that 
\begin{equation}
\label{GWforGrass}
d_r\left( C(\ell ,m)\right) \ge q^{\d} + q^{\d - 1} +\dots +   q^{\d - r +1} \quad \text{where} \quad \d := 
\ell(m-\ell).
\end{equation}
The latter is a consequence of the so called Griesmer-Wei bounds for linear codes and a result of Nogin \cite{N} which says that $d_1=q^{\d}$. In fact, 
Nogin \cite{N} showed that the Griesmer-Wei bound is sometimes attained, that is, 
\begin{equation}
\label{uptomu}
d_r\left( C(\ell ,m)\right) = q^{\d} + q^{\d - 1} +\dots +   q^{\d - r +1}
\quad {\rm for} \quad 0\le r \le \mu,
\end{equation} 
where
$$
\mu:= \max \{\ell, m - \ell \} + 1.
$$
Alternative proofs of Nogin's result for higher weights of $C(\ell,m)$ were given by Ghorpade and Lachuad \cite{GL} using the notion of a closed family. Recently, Hansen, Johnsen and Ranestad \cite{HJR} have observed that a dual result holds as well, namely, 
\begin{equation}
\label{dualweights}
d_{k-r}\left( C(\ell ,m)\right) = n -( 1+ q + \cdots + q^{r -1})
\quad {\rm for} \quad 0\le r \le \mu.
\end{equation}
In general, the values of $d_r(C(\ell,m))$ for $\mu < r < k -\mu$ are not known. For example,
if $\ell =2$ and we assume (without loss of generality) that $m\ge 4$, then $\mu = m-1$, and 
$d_r(C(\ell,m))$ for  $m \le r < {{m-1}\choose{2}}$ are not known, except that in the first
nontrivial case,  Hansen, Johnsen and Ranestad \cite{HJR} have
shown by clever algebraic-geometric arguments that
\begin{equation}
\label{d5C2m}
d_5(C(2,5)) =  q^{6} +q^{5} + 2q^4 + q^{3} = d_4 + q^4.
\end{equation}
Notice that the Griesmer-Wei bound in \eqref{GWforGrass} is not attained in this case. Nonetheless,  Hansen, Johnsen and Ranestad \cite{HJR} conjecture that 
the difference $d_r - d_{r-1}$ of consecutive higher weights of $C(\ell,m)$ is always 
a power of $q$. 

Our main results concerning the the determination of $d_r(C(\ell,m))$ are as follows. 
First, we recover \eqref{uptomu} and \eqref{dualweights} as an immediate corollary of 
our characterization of decomposable subspaces. Next, we further analyze the structure of   decomposable vectors in $\bigwedge^{2}V$ to extend \eqref{dualweights} by showing that
\begin{equation}
\label{dualnewdr}
d_{k-\mu-1}\left(C(2,m)\right) =
 n-(1+q+\cdots+q^{\mu-1}+q^2) = d_{k-\mu} - q^2
  \quad \mbox{for any } m \ge 4. 
\end{equation}
Finally, we use the abovementioned analysis of decomposable vectors in $\bigwedge^{2}V$ and also exploit the Hodge star duality to prove the following generalization of \eqref{d5C2m} for any $m\ge 4$. 
\begin{equation}
\label{newdr}
d_{\mu + 1}\left(C(2,m)\right) =  q^{\d} + q^{\d - 1}+ 2q^{\d - 2}+ q^{\d - 3} +\dots +   q^{\d - \mu+1} = d_\mu+q^{\delta-2}.
\end{equation}
In the course of deriving these formulae, we use a mild generalization of the Griesmer-Wei bound,  proved here in the general context of arbitrary linear codes, which may be of independent interest.

It is hoped that these results, and more so, the methods used in proving them, will pave the way for the solution of the problem of determination of the complete weight hierarchy of $C(\ell, m)$ at least in the case $\ell =2$.  
To this end, we provide, toward the end of this paper,  an initial tangible goal by stating  conjectural formulae for $d_r\left(C(2,m)\right)$ when $\mu +1 \le r \le 2\mu -3$, and also when $k-2\mu +3\le r \le k- \mu -1$. It may be noted that these conjectural formulae, and of course both \eqref{dualnewdr} and \eqref{newdr}, corroborate the conjecture of Hansen, Johnsen and Ranestad \cite{HJR} that the differences of consecutive higher weights of Grassmann codes is always a power of $q$.

\section{Decomposable Subspaces}
\label{sec:cd}

Let us fix, in this as well as the next section, positive integers $\ell, m$ with $\ell \le m$, a field $F$, and a vector space $V$  of dimension $m$ over $F$.  
Let 
$$
I(\ell,m):=\{ \a = (\a_1, \dots , \a_{\ell} )\in \Z^{\ell}  : 
1\le \a_1 < \dots < \a_\ell  \le m \} .
$$
If $\{v_1, \dots , v_m\}$ is a basis of $V$, then 
$\{v_{\a} : \a\in I(\ell,m)\}$ is a basis of $\bigwedge^{\ell}V$, where 
$v_{\a} := v_{\a_1}\wedge \cdots \wedge v_{\a_{\ell}}$.
Given any $\omega \in \bigwedge^{\ell}V$, define 
$$
V_{\omega} : = \{v\in V : v \wedge \omega = 0\}.
$$
Clearly, $V_{\omega}$ is a subspace of $V$. It is evident that $\omega =0$ if and only if $\dim V_{\omega} =m$. The following elementary characterization will be useful in the sequel. Here, and hereafter, it may be useful to keep in mind 
that for us, a   decomposable vector is necessarily nonzero.

\begin{lemma}
\label{elemlemma1}
Assume that $\ell < m$ and let $\omega\in \bigwedge^{\ell}V$. Then
$$
\omega \text{ is   decomposable } \Longleftrightarrow \dim V_{\omega} = \ell.
$$
Moreover, if $\dim V_{\omega} = \ell$ and $\{v_1, \dots , v_{\ell}\}$ is a basis of $V_{\omega} $, then $\omega = c(v_1\wedge \cdots \wedge v_{\ell})$ for some $c\in F$ with $c\ne 0$.
\end{lemma}

\begin{proof}
If $\omega$ is   decomposable, then $\omega = v_1\wedge \cdots \wedge v_{\ell}$ for some linearly independent elements $v_1, \dots , v_{\ell}\in V$. Clearly, $\{v_1, \dots , v_{\ell}\} \subseteq V_{\omega}$. Moreover, if $v\in V_{\omega}$, then $v, v_1, \dots , v_{\ell}$ are linearly dependent. It follows that $\{v_1, \dots , v_{\ell}\}$ is a basis of $V_{\omega} $. 
Conversely, let $\dim V_{\omega}=\ell$. Extend a basis  $\{v_1, \dots , v_{\ell}\}$ of $V_{\omega}$ to a basis $\{v_1, \dots , v_m\}$ of $V$. 
Write $\omega = \sum_{\alpha\in I(\ell,m)} c_{\a}\, v_{\a}$. 
Now $0= v_i \wedge \omega =\sum_{\alpha\in I(\ell,m)} c_{\a}\, (v_i\wedge v_{\a})$ for $\ell < i\le m$. Consequently, $c_{\a} =0$ if $i$ does not appear in $\a$.  It follows that $\omega = c_{(1,2,\dots \ell)}\left(v_1\wedge \cdots \wedge v_{\ell}\right)$, as desired. 
\end{proof}

\begin{corollary}
\label{corlem1}
If $\ell =1$ or $\ell = m-1$, then the space $\bigwedge^{\ell}V$ is decomposable, that is, every nonzero element of $\bigwedge^{\ell}V$ is   decomposable.
\end{corollary}

\begin{proof}
The result is obvious when $\ell =1$. Suppose $\ell = m-1$.  
Now $\bigwedge^{m}V$ is canonically isomorphic to $F$, and for $0\ne \omega \in \bigwedge^{\ell}V$, the linear  map from $V$ to $F$ given by $v\mapsto v\wedge \omega$ is nonzero and hence surjective. Clearly, $V_{\omega} $ is the kernel of this linear map and so $\dim V_{\omega} =\dim V - 1 = \ell$. Thus, Lemma \ref{elemlemma1} applies. 
\end{proof}

\begin{lemma}
\label{structure_lemma1}
Let $\omega_1, \omega_2 \in \bigwedge^{\ell}V$ be   decomposable and 
linearly independent, and let $V_i = V_{\omega_i}$ for $i=1,2$. Then
$$
\omega_1+\omega_2 \text{ is   decomposable } \Longleftrightarrow 
\dim V_{1}\cap V_{2} = \ell - 1 \Longleftrightarrow 
\dim V_{1} + V_{2} = \ell + 1.
$$
\end{lemma}

\begin{proof} 
Assume that $\dim V_{1}\cap V_{2} = \ell - 1$.
Let $\{f_1,\dots,f_{\ell-1}\}$ be a basis for $V_{1} \cap V_{2}$. 
Extend it to bases $\{f_1,\dots,f_{\ell-1},g_1\}$
and $\{f_1,\dots,f_{\ell-1},g_2 \}$ of $V_1$ and $V_2$, respectively. 
By Lemma \ref{elemlemma1}, there are 
$c_1,c_2\in F$ such that $\omega_i = c_i(f_1\wedge f_2\wedge \cdots\wedge f_{\ell-1}\wedge g_i)$ for $i=1,2$. Now $\omega_1+\omega_2\ne 0$ since $\omega_1,\omega_2$ are linearly independent, and $\omega_1+\omega_2=f_1\wedge f_2\wedge \cdots\wedge f_{\ell-1}\wedge (c_1g_1+c_2g_2)$. Thus $\omega_1+\omega_2$ is   decomposable. 

Conversely, suppose $\omega_1+\omega_2$ is   decomposable. 
Let $W= V_{\omega_1+\omega_2}$. It is clear that $V_1\cap V_2\subseteq W$. 
Also, by Lemma \ref{elemlemma1}, $\dim W = \ell = \dim V_1=\dim V_2$.  
Hence if $V_1\cap V_2=W$, then $V_1\cap V_2= V_1 = V_2$, which contradicts the linear independence of $\omega_1$ and $\omega_2$. Thus, $\dim V_1\cap V_2 \le \ell -1$, or equivalently, $\dim V_1+V_2 \ge \ell +1$. Moreover, we can find $z\in W\setminus (V_1\cap V_2)$. Note that since $z\wedge (\omega_1+\omega_2)=0$, we have: 
$z\wedge \omega_1 = 0 \Leftrightarrow z\wedge \omega_2 =0$. Hence $z\not\in V_1\cup V_2$ and $V_i+Fz$ has dimension $\ell+1$ for $i=1,2$. Further, since $z\wedge \omega_1 = - z\wedge \omega_2$, in view of Lemma \ref{elemlemma1} we see that $V_1+Fz = V_{z\wedge \omega_1} = V_{z\wedge \omega_2} = V_2+Fz$. 
Consequently, $V_1+V_2 \subseteq V_1+Fz=V_2+Fz$ and $\dim V_1+V_2 \le \ell +1$. This proves that $\dim V_1+V_2 = \ell +1$ or equivalently, 
$\dim V_1\cap V_2 = \ell -1$. This proves the desired equivalence.
\end{proof}

\begin{corollary}
\label{lemma_ab}
Let  $v_1,v_2,v_3,v_4 \in V$ and suppose 
$\omega:= (v_1\wedge v_2)+(v_3 \wedge v_4) \in \bigwedge^2V$ is nonzero. 
Then $\omega$  is   decomposable if and only if $\{v_1,v_2,v_3,v_4\}$  is 
linearly dependent.
\end{corollary}

\begin{proof} 
When $v_1\wedge v_2$ and $v_3 \wedge v_4$ are linearly independent, the result follows 
from Lemma \ref{structure_lemma1}. The case when  $v_1\wedge v_2$ and $v_3 \wedge v_4$ are linearly dependent is easy. 
\end{proof}

Given a subspace $E$ of $\bigwedge^{\ell}V$, let us define
$$
V_E :=\bigcap_{\omega\in E} V_{\omega} \quad \mbox{ and } \quad 
V^E:= \sum_{0\ne \omega\in E} V_{\omega} .
$$ 
Now, let $r=\dim E$. We say that the subspace $E$ is {\em close of type I\,} if there are $\ell + r-1$ 
linearly independent elements $f_1, \dots, f_{\ell-1}, g_1, \dots , g_r$ 
in $V$ such that 
$$
E=\vspan\{f_1\wedge \dots \wedge f_{\ell-1}\wedge g_i : i=1, \dots, r\}.
$$ 
And we say that $E$ is {\em close of type II\,} if there are $\ell + 1$ 
linearly independent elements $u_1, \dots, u_{\ell-r+1}, g_1, \dots , g_r$ 
in $V$ such that 
$$
E=\vspan\{u_1\wedge \dots \wedge u_{\ell-r+1}\wedge g_1 \cdots 
\wedge \check{ g_i} \wedge  \dots \wedge g_r  : i=1, \dots, r\},
$$
where $\check{ g_i}$ indicates that $g_i$ is deleted. We 
say that $E$ is a {\em close subspace} of $\bigwedge^{\ell}V$ if 
$E$ is close of type I or close of type II. 

Evidently, every one-dimensional subspace of $\bigwedge^{\ell}V$ is close of type I as well as of type II, whereas for two-dimensional subspaces, the 
notions of close subspaces of type I and type II are identical. A corollary of the following lemma is that in dimensions three or more, 
the two notions are distinct and mutually disjoint.

\begin{lemma}
\label{lem:VsubEandsupE}
Let $E$ be  a close subspace of $\bigwedge^{\ell}V$ of dimension $r$. Then
$E$ is   decomposable. Moreover, if  $\{\omega_1, \dots , \omega_r\}$ is a basis of $E$, then $V_E= V_{\omega_1} \cap \cdots \cap V_{\omega_r}$ and $V^E= V_{\omega_1}+\cdots + V_{\omega_r}$. Further, assuming that $r>1$, we have  $\dim V_E = \ell -1$ and $\dim V^E =  \ell+r-1$ if $E$ is close of type I, whereas $\dim V_E = \ell -r+ 1$ and $\dim V^E =  \ell+1$ if $E$ is close of type II.
\end{lemma}

\begin{proof}  
Since $E$ is close of dimension $r$,  there is a $r$-dimensional subspace $G$ of $V$ [in fact, $G=\vspan\{g_1, \dots , g_r\}$, in the above notation] such that $E$ is naturally isomorphic to  $\bigwedge^1G$ or to $\bigwedge^{r-1}G$ according as $E$ is close of type I or of type II. Thus, in view of Corollary \ref{corlem1}, we see that  $E$ is   decomposable. Next, suppose
$\{\omega_1, \dots , \omega_r\}$ is a basis of $E$. Then obviously, 
$V_E= V_{\omega_1} \cap \cdots \cap V_{\omega_r}$. Moreover, in view of  Lemmas \ref{elemlemma1} and \ref{structure_lemma1}, we see that 
$V_{\omega + \omega'}\subseteq V_{\omega} + V_{\omega'}$ for all nonzero $\omega, \omega'\in E$ such that 
$\omega + \omega'\ne 0$. Hence, by induction on $r$, we obtain $V^E= V_{\omega_1}+\cdots + V_{\omega_r}$. Finally, suppose $r>1$. In case $E$ is close of type I, and $f_1, \dots, f_{\ell-1}, g_1, \dots , g_r$ are linearly independent elements of $V$ as in the definition above, then 
in view of Lemma \ref{elemlemma1}, we see that 
$V_E= \cap_{i=1}^r \vspan\{f_1, \dots, f_{\ell-1}, g_i\} = \vspan\{f_1, \dots, f_{\ell-1}\}$ and $V^E = \sum_{i=1}^r \vspan\{f_1, \dots, f_{\ell-1}, g_i\} = \vspan\{f_1, \dots, f_{\ell-1},g_1, \dots, g_r\}$. On the other hand, 
if $E$ is close of type II, and $u_1, \dots, u_{\ell-r+1}, g_1, \dots , g_r$ are linearly independent elements of $V$ as in the definition above, then 
as before, in view of Lemma \ref{elemlemma1}, we see that 
$V_E= \vspan\{u_1, \dots, u_{\ell-r+1}\}$ and $V^E = \vspan\{u_1, \dots, u_{\ell-r+1}, g_1, \dots, g_r\}$. This proves the desired assertions about 
$\dim V_E$ and $\dim V^E$.
\end{proof} 

The following result may be compared with  \cite[Thm. 4.2]{GL2}. Also, the proof is structurally analogous to that of  \cite[Thm. 4.2]{GL2}, except that the arguments here are a little more subtle. 

\begin{theorem}[Structure Theorem for Decomposable Subspaces]
\label{structure}
A subspace of $\bigwedge^{\ell}V$ is decomposable if and only if it is close. 
\end{theorem}

\begin{proof}
Lemma \ref{lem:VsubEandsupE} proves that a close subspace of $\bigwedge^{\ell}V$ is   decomposable. To prove the converse, let $E$ be a   decomposable subspace of $\bigwedge^{\ell}V$. We induct on $r:=\dim E$. 
The case $r=1$ is trivial, whereas if $r=2$, then the desired result follows from Lemmas \ref{elemlemma1} and \ref{structure_lemma1}. Now, suppose $r=3$. 
Let $\{\omega_1, \omega_2, \omega_3\}$ be a basis of $E$, and let $V_i=V_{\omega_i}$ for $i=1,2,3$.  
Then $\dim V_i = \ell$ and $\dim V_i\cap V_j = \ell -1$ for $1\le i, j\le 3$ with $i\ne j$, thanks to Lemmas \ref{elemlemma1} and \ref{structure_lemma1}. Thus, if we let $W=V_1\cap V_2\cap V_3$,
then 
$$
\ell -2 = \dim V_1\cap V_2 + \dim V_1\cap V_3 -\dim V_1 \le \dim W \le \dim V_1\cap V_2  = \ell -1. 
$$ 
If $\dim W = \ell -1$, then we can find $\ell +2$ elements $f_1, \dots, f_{\ell -1}, g_1, g_2, g_3$ in $V$ such that $\{f_1, \dots, f_{\ell -1}\}$ is a basis of $W$ and 
$\{f_1, \dots, f_{\ell -1}, g_i\}$ is a basis of $V_i$ for $i=1,2,3$. 
We may assume without loss of generality that 
$\omega_i = f_1\wedge \dots \wedge f_{\ell -1}\wedge g_i$ for $i=1,2,3$, thanks to Lemma \ref{elemlemma1}. Since $\omega_1, \omega_2, \omega_3$ are linearly independent, it follows that $g_i\not\in \sum_{j\ne i} V_j$ for $i=1,2,3$. Consequently, 
$f_1, \dots, f_{\ell -1}, g_1, g_2, g_3$ are linearly independent elements of $V$ and $E$ is close of type~I. On the other hand, if $\dim W = \ell -2$, 
then we can find $\ell +1$ elements $u_1, \dots, u_{\ell -2}, g_1, g_2, g_3$ in $V$ such that $\{u_1, \dots, u_{\ell -2}\}$ is a basis of $W$, and 
$\{u_1, \dots, u_{\ell -2}, g_i\}$ is a basis of $\cap_{j\ne i} V_j$, and moreover, $g_i\not\in V_i$ for $i=1,2,3$. Consequently, $u_1, \dots, u_{\ell -2}, g_1, g_2, g_3$ are linearly independent elements of $V$ [indeed, the vanishing of a linear combination of $u_1, \dots, u_{\ell -2}, g_1, g_2, g_3$ in which the coefficient of $g_i$ is nonzero implies that $g_i$ is in $V_i$]. Hence in view of Lemma \ref{elemlemma1}, we see that for $i=1,2,3$, the set  
$\{u_1, \dots, u_{\ell -2}, g_1, g_2, g_3\}\setminus \{g_i\}$ is a basis of
$V_i$ and $\omega_i = c_i\left( u_1\wedge \dots \wedge u_{\ell-2}\wedge g_{i_1} \wedge g_{i_2}  \right)$ for some $c_i\in F\setminus \{0 \}$, where $1\le i_1<i_2\le 3$ with $i_1\ne i \ne i_2$. 
It follows that $E$ is close of type II.

Finally, we assume that $r>3$ and that every   decomposable subspace of dimension $<r$ is close of type I or of type II. Let $\{\omega_1, \dots , \omega_r\}$ be a basis of $E$, and let $V_i=V_{\omega_i}$ 
and $E_i =\vspan\{\omega_1, \dots ,\omega_{i-1}, \omega_{i+1}, \dots,  \omega_r\}$ for $i=1,\dots , r$. 
Each $E_i$ is   decomposable and by the induction hypothesis, we are in one of the following two cases. 
\smallskip

{\em Case 1:} $E_i$ is close of type I for some $i\in\{1,\dots ,r\}$. 
\smallskip

Fix $i\in\{1,\dots ,r\}$ such that $E_i$ is close of type I, and let $W_i:=V_{E_i} = \cap_{j\ne i}V_j$. Then $\dim W_i = \ell -1$ and since $V_E = W_i\cap V_i$, by picking any $j\in\{1,\dots ,r\}$ with $j\ne i$, and using Lemma \ref{structure_lemma1}, we find 
$$
\ell -2 
= \dim W_i + \dim V_i -\dim (V_j + V_i) \le \dim V_E \le \dim W_i = \ell -1, 
$$
If $\dim V_E = \ell -1$, then it is readily seen that $E$ is close of type I. 
Suppose, if possible, $\dim V_E = \ell -2$. Let $\{f_1, \dots, f_{\ell-2}\}$ be a basis of $V_{E}$ and $f_{\ell-1}$ be any element of $W_i\setminus V_E$. Then $\{f_1, \dots, f_{\ell-1}\}$ is a basis of $W_i$ and 
$f_{\ell-1}\not\in V_i$. Since $\dim V_j\cap V_i = \ell -1$ for $j\ne i$, we can find $g_1, \dots , g_{i-1}, g_{i+1}, \dots , g_{r}\in V_i$ such that $\{f_1, \dots, f_{\ell-2}, g_j\}$ is a basis of $V_j\cap V_i$ for $j\in\{1,\dots ,r\}$ with $j\ne i$. Also, since $f_{\ell-1}\in W_i\setminus V_E$, we see that  $\{f_1, \dots, f_{\ell-1}, g_j\}$ is a basis of $V_j$, and so by Lemma \ref{elemlemma1}, 
each $\omega_j$ is a nonzero scalar multiple of $f_1 \wedge \dots \wedge f_{\ell-1} \wedge g_j$ for $j\in\{1,\dots ,r\}$ with $j\ne i$.
Now $\omega_1, \dots , \omega_{i-1}, \omega_{i+1}, \dots,  \omega_r$ are linearly independent elements of $\bigwedge^{\ell}V$, and therefore 
$f_1, \dots, f_{\ell-1}, g_1, \dots , g_{i-1}, g_{i+1}, \dots , g_{r}$ are linearly independent of $V$. In particular, the $\ell$-dimensional space $V_i$ contains $\ell+r-3$ linearly independent elements $f_1, \dots, f_{\ell-2}, g_1, \dots , g_{i-1}, g_{i+1}, \dots , g_{r}$, which is a contradiction since $r>3$. Thus we have shown that $E$ is close of type in I.
\smallskip

{\em Case 2:} $E_i$ is close of type II for each $i\in \{1,\dots, r\}$. 
\smallskip

In this case each $W_i:= V_{E_i}$ is of dimension  $\ell - r+2$ and as before, 
picking any $j\in\{1,\dots ,r\}$ with $j\ne i$, and 
using Lemma \ref{structure_lemma1}, we find 
$$
\ell - r +1 
= \dim W_i + \dim V_i -\dim (V_j + V_i) \le \dim V_E \le \dim W_i  = \ell - r +2.
$$
First, suppose $\dim V_E = \ell - r +1$. Fix a basis $\{u_1, \dots, u_{\ell-r+1}\}$ of $V_E$. For each $i\in\{1,\dots ,r\}$, choose $g_i\in W_i\setminus V_E$. Then $g_i\not\in V_i$ and $\{u_1, \dots, u_{\ell-r+1}, g_i\}$ is a basis of $W_i$ for $i=1,\dots , r$. Observe that the $\ell+1$ elements $u_1, \dots, u_{\ell-r+1}, g_1, \dots , g_r$ of $V$ are linearly independent [indeed, the vanishing of a linear combination of $u_1, \dots, u_{\ell-r+1}, g_1, \dots , g_r$ in which the coefficient of $g_i$ is nonzero implies that $g_i$ is in $V_i$]. 
Hence the subset $\{u_1, \dots, u_{\ell-r+1}, g_1, \dots , g_{i-1}, g_{i+1}, \dots ,  g_r \}$ 
of $V_i$ is 
a basis of $V_i$, and so 
in view of Lemma \ref{elemlemma1}, we see that 
$\omega_i$ is a nonzero scalar multiple of $u_1\wedge \dots \wedge u_{\ell-r+1}\wedge g_1 \cdots 
\wedge \check{ g_i} \wedge  \dots \wedge g_r$ for $i=1,\dots , r$. 
Thus we have shown that if $\dim V_E = \ell - r +1$, then $E$ is close of type II. 
Now suppose, if possible, $\dim V_E = \ell - r +2$. Then $V_E=W_i$ for $i=1,\dots , r$. 
Since $r>3$ and $E_{r-1}$ is close of type II, we see that the subspace $E^*:=\vspan\{\omega_1, \omega_2,\omega_{r}\}$ of $E_{r-1}$ is close of type II.  In particular, $\dim V_{E^*} = \dim V_1\cap V_2\cap V_r = \ell -2$.  
Thus, in view of Lemma \ref{structure_lemma1}, we see that  
$\dim (V_1+V_2+V_r)$ is at most
$$
\dim V_1 + \dim V_2 + \dim V_r - \dim V_1\cap V_2 - 
\dim V_1\cap V_r - \dim V_2\cap V_r + \dim V_1\cap V_2\cap V_r, \\
$$
which is $3\ell - 3(\ell -1) + (\ell -2) = \ell +1$. 
Also, by Lemma \ref{structure_lemma1}, we have 
$\dim (V_1+V_2+V_r)  \ge \dim (V_1+V_2) = \ell +1$. It follows that 
$ V_1+V_2+V_r = V_1 + V_2$, or equivalently, $V_r\subseteq V_1+V_2$. 
We will now use this to arrive at a contradiction. To this end, consider the space $E_r$. 
Since $E_r$ is close of type II, we can find $\ell+1$ linearly independent elements $u_1, \dots, u_{\ell-r+2}, g_1, \dots , g_{r-1}$ in $V$ such that $u_1, \dots, u_{\ell-r+2}$ span $W_r$ and $\omega_i = u_1\wedge \dots \wedge u_{\ell-r+2}\wedge g_1 \cdots 
\wedge \check{ g_i} \wedge  \dots \wedge g_{r-1}$ for $i=1,\dots , r-1$.  It is clear that $\{u_1, \dots, u_{\ell-r+2}, g_1, \dots , g_{r-1}\}$ is a basis of $V_1+V_2$. Also, since $W_r = V_E$, we can add $r-2$ elements to the set $\{u_1, \dots, u_{\ell-r+2}\}$ to obtain a basis of $V_r$. But, $V_r\subseteq V_1+V_2$ and so the additional $r-2$ basis elements of $V_r$ are linear combinations of $u_1, \dots, u_{\ell-r+2}, g_1, \dots , g_{r-1}$.
Consequently, $\omega_r$ is a linear combination of $\omega_1, \dots , \omega_{r-1}$, which is a contradiction. 
\end{proof}
 
\begin{corollary}
\label{cor:struturethm}
Let $\mu:=\max\{\ell, m-\ell\}+1$ and $r$ be any positive integer. Then $\bigwedge^{\ell}V$ has a   decomposable subspace of dimension $r$
if and only if $r\le \mu$. Moreover, a close subspace of type I (resp: type II) of dimension $r$ exists if and only if $r\le m-\ell +1$ (resp: $r\le \ell +1$).
\end{corollary}

\begin{proof}  
Let $E$ be a subspace of $\bigwedge^{\ell}V$ of dimension $r$. 
By Lemma \ref{lem:VsubEandsupE}, if $E$ is  close of type I, 
then $\ell + r - 1 = \dim V^E \le m$,  that is, $r\le m-\ell +1$, whereas if $E$ is close of type II, then $\ell-r+ 1 = \dim V_E \ge 0$, that is, $r\le \ell +1$. Thus, Theorem \ref{structure} implies that if  $\bigwedge^{\ell}V$ has a   decomposable subspace of dimension $r$, then $r \le \mu$. The converse is an immediate consequence of the definition of close subspaces and their decomposability. 
\end{proof}

\begin{remark}
Decomposable vectors in $\bigwedge^{\ell}V$ are variously known as \emph{pure $\ell$-vectors} (cf. \cite[\S 11.13]{bourbaki}), \emph{extensors of step $\ell$} (cf. \cite[\S 3]{BBR}), or \emph{completely decomposable vectors} (cf. \cite{N}). Some of the preliminary lemmas proved initially in this section are not really new. For example, Lemma \ref{elemlemma1} appears essentially as Exercise 17 (a) in Bourbaki \cite[p. 650]{bourbaki} or as Theorem 1.1 in Marcus \cite{marcus}, Corollary \ref{corlem1} is basically Theorem 1.3 of \cite{marcus}, and Lemma \ref{structure_lemma1} is  
a consequence of Exercise 17 (c) in \cite[p. 651]{bourbaki}. We have stated these results in a form convenient for our purpose, and included the proofs for the sake of completeness. At any rate,  as far as we know, Theorem \ref{structure} is new. On the other hand, characterization of decomposable subspaces has been studied in the setting of symmetric algebras. Although one comes across subspaces of various types, including those similar to the ones considered in this section, the situation for subspaces of symmetric powers is rather different and the characteristic of the underlying field plays a role. We refer to the papers of Cummings \cite{cummings} and Lim \cite{lim} for more on this topic. In the context of tensor algebras, the opposite of decomposable subspaces has been considered, namely, \emph{completely entangled} subspaces wherein no nonzero element is decomposable. A neat formula for the maximum possible dimension of completely entangled subspaces of the tensor product of finite dimensional complex vector spaces is given by Parthasarathy  \cite{krp}. As remarked earlier, determining the structure of decomposable subspaces corresponds to determining the linear subvarieties in the Grassmann variety $G_{\ell,m}$. A special case of this has been considered, in a similar, but more general, geometric setting by Tanao \cite{tanao}, where subvarieties of $G_{2,m}$ biregular to $\PP^m$ over an algebraically closed field of characteristic zero are studied. 
\end{remark}

\section{Duality and the Hodge Star Operator}
\label{sec:duality}

We have seen in Section \ref{sec:cd} that a   decomposable subspace of $\bigwedge^{\ell}V$ is close of type I or of type II. It turns out that the two types are dual to each other. 
This is best described using the so called Hodge star operator
$
\hodge: {\bigwedge}^{\ell}V \to {\bigwedge}^{m-\ell}V,
$
which may be defined as follows. Fix an ordered basis $\{e_1, \dots , e_m\}$ of $V$ and use it to identify $\bigwedge^{m}V$ with $F$ so that $e_1\wedge\cdots \wedge e_m = 1$. Let $I(\ell,m)$ and $e_{\a}$ for $\a\in I(\ell,m)$ be as in Section \ref{sec:cd}. Moreover, for $\a=(\a_1, \dots , \a_{\ell})\in I(\ell,m)$, let $\a^c= (\a^c_1, \dots , \a^c_{m-\ell})$ denote the unique element of $I(m-\ell,m)$ such that 
$\{\a_1, \dots , \a_{\ell}\}\cup \{\a^c_1, \dots , \a^c_{m-\ell}\} = \{1,\dots, m\}$. 
Then $\hodge: \bigwedge^{\ell}V \to \bigwedge^{m-\ell}V$ is the unique $F$-linear map satisfying
$$
\hodge (e_{\a} ) = (-1)^{\a_1+\cdots + \a_{\ell} + \ell(\ell+1)/2} \, e_{\a^c} \quad \mbox{ for } \a\in I(\ell,m).
$$ 
Clearly, $\hodge$ is a vector space isomorphism. 
The key property of $\hodge$ is that it is 
essentially independent of the choice of ordered basis of $V$, and as such, it maps   decomposable elements in $\bigwedge^{\ell}V$ to   decomposable elements in $\bigwedge^{m-\ell}V$. (See, for example, \cite[Sec. 6]{BBR} and \cite[Sec. 4.1]{marcus}.) In particular,   decomposable subspace of $\bigwedge^{\ell}V$ are mapped to   decomposable subspaces of  $\bigwedge^{m-\ell}V$. 
Moreover, it is easy to see that via the Hodge star operator, 
close subspaces of type I are mapped  to close subspaces of  type II, whereas
close subspaces of type II are mapped  to close subspaces of  type I. 
Thus, the two types are dual to each other. 
 
In the case $\ell =2$, both $\bigwedge^{\ell}V$ and $\bigwedge^{m-\ell}V$ are closely related to the space $B_m$ of all $m\times m$ skew-symmetric matrices with entries in $F$, and the relation is compatible with the Hodge star operator. To state this a little more formally, we introduce some terminology below and make a few useful observations. In the remainder of this section we tacitly assume that $m>2$. 

Given any $u\in V$, let $\bfu$ denote the $m\times 1$ column vector whose entries are the coordinates of $u$ with respect to the ordered basis $\{e_1, \dots , e_m\}$. In particular, $\e_i$ has $1$ as its  $i$th entry and all other entries are $0$.  
Consider the $F$-linear maps
$$
\sigma : {\bigwedge}^{2}V \to B_m \quad \mbox{ and }\quad  \pi : {\bigwedge}^{m-2}V \to B_m 
$$
defined by 
$$
\sigma(e_r\wedge e_s) = \e_r\e_s^t - \e_s\e_r^t \mbox{ for  } 1\le r<s\le m \quad \mbox{ and } \quad  \pi (\omega) = A_{\omega} \mbox{ for } \omega\in {\bigwedge}^{m-2}V,
$$
where $\e^t$ denotes the transpose of $\e$  
and $A_{\omega}$ denotes the $m\times m$ matrix whose $(i,j)$th entry is (the unique scalar corresponding to) $e_i\wedge e_j \wedge \omega$. 

\begin{lemma}
\label{lem:compat}
$\sigma = \pi \circ \hodge$.
\end{lemma}

\begin{proof}  
We have 
$\hodge (e_r \wedge e_s) = (-1)^{r+s+1} \left(e_1\wedge e_2\wedge \cdots \wedge \check{e_r}\wedge \cdots \wedge\check{e_s}\wedge\cdots\wedge e_m\right)$ 
for $1\le r<s\le m$,
where  $\check{ }$ indicates that the corresponding entry is removed. Now,
$$
e_i \wedge e_j \wedge \left(e_1\wedge e_2\wedge \cdots \wedge \check{e_r}\wedge \cdots \wedge\check{e_s}\wedge\cdots\wedge e_m\right) = \left\{ 
\begin{array}{cl} 
(-1)^{i+j-3} & \mbox{ if } (r,s) = (i,j), \\
(-1)^{i+j-2} & \mbox{ if } (r,s) = (j,i), \\
0 & \mbox{ otherwise }  \\\end{array} \right.
$$
for $1\le i,j,r,s \le m$ with $r<s$. 
It follows that  
$ \pi \circ \hodge(e_r \wedge e_s) = \e_r\e_s^t - \e_s\e_r^t = \sigma (e_r \wedge e_s)$
for $1\le r<s\le m$. Since $\{e_r \wedge e_s : 1\le r<s\le m\}$ is a basis of 
${\bigwedge}^{2}V$ and all the maps are linear, the lemma is proved.
\end{proof}

Given any $\omega'\in {\bigwedge}^{2}V$ and $\omega\in {\bigwedge}^{m-2}V$, we refer to the rank of $\sigma(\omega')$ [resp: $\pi(\omega)$] as the {\em rank} of $\omega'$ [resp: $\omega$], and denote it by $\rank(\omega')$ [resp: $\rank(\omega)$]. 
Note that if $\omega= \hodge(\omega')$, then $\rank(\omega') = \rank\left(\omega\right)$, thanks to Lemma \ref{lem:compat}. 

\begin{corollary}
\label{cor:sigmapirank2}
Both $\sigma$ and $\pi$ are vector space isomorphisms. Moreover, 
\begin{equation}
\label{decomprank2}
\omega' \mbox{ is decomposable } \Longleftrightarrow \rank (\omega') = 2 
\qquad \mbox{ for any } \omega' \in {\bigwedge}^{2}V,
\end{equation}
and
\begin{equation}
\label{dualdecomprank2}
\omega \mbox{ is decomposable } \Longleftrightarrow \rank (\omega) = 2
\qquad \mbox{ for any } \omega \in {\bigwedge}^{m-2}V.
\end{equation}
\end{corollary}

\begin{proof}
It is evident that $\sigma$ is an isomorphism. Hence by Lemma \ref{lem:compat}, so is $\pi$. 
Now,  given any $\omega\in {\bigwedge}^{m-2}V$, the kernel of (the linear map from $V$ to $V$ corresponding to) $\pi(\omega) = A_\omega$ is the space $V_{\omega}$. 
Hence \eqref{dualdecomprank2} follows from Lemma \ref{elemlemma1}. Next, if $\omega' \in {\bigwedge}^{2}V$ is decomposable, then $\omega'=u\wedge v$ for some $u,v\in V$ and $\sigma(\omega') = \bfu\bfv^t -\bfv\bfu^t$. It follows that $\sigma(\omega')$ is of rank $2$. 
This proves the implication $\Longrightarrow$ in \eqref{decomprank2}. The other implication follows from \eqref{dualdecomprank2} together with Lemma \ref{lem:compat} and the fact $\hodge$ gives a one-to-one correspondence between decomposable elements. 
\end{proof} 

\begin{corollary}
\label{cor:sigmarank4}
Let  $v_1,v_2,v_3,v_4 \in V$ and suppose 
$\omega:= (v_1\wedge v_2)+(v_3 \wedge v_4) \in \bigwedge^2V$ is nonzero. Then the rank of $\sigma(\omega)$ is $2$ or $4$ according as the set $\{v_1,v_2,v_3,v_4\}$  is 
linearly dependent or linearly independent.
\end{corollary}

\begin{proof} 
Follows from  \eqref{decomprank2} above and Corollary \ref{lemma_ab} in view of the fact that 
a skew-symmetric matrix is always of even rank. 
\end{proof}

\section{Griesmer-Wei Bound and its Generalization}
\label{sec:Griesmer}

Let us begin by reviewing some generalities about (linear, error correcting) codes. 
Fix integers $k,n$ with $1\le k \le n$ and a prime power $q$. 
Let $C$ be a linear $[n,k]_q$-code, i.e., 
let $C$ be a $k$-dimensional subspace of the $n$-dimensional vector space $\Fq^n$ 
over the finite field $\Fq$ with $q$ elements. 
Given any $x=(x_1,\dots,x_n)$ in $\Fq^n$, let
$$
\supp (x):=\{ i : x_i\ne 0\} \quad \mbox{ and } \quad \Vert x \Vert := |\supp (x)|
$$
denote the {\em support} and the {\em (Hamming) norm} of $x$. 
More generally, for $D\subseteq \Fq^n$, 
let
$$
\supp (D):=\{ i : x_i\ne 0 \mbox{ for some } x=(x_1,\dots ,x_n) \in D\} 
\quad \mbox{ and } \quad \Vert D \Vert := |\supp (D)|
$$
denote the {\em support} and the {\em (Hamming) norm} of $D$. The {\em minimum distance} of $C$ is defined by 
$d(C):=\min\{\Vert x \Vert : x\in C \mbox{  with } x\ne 0\}$. 
More generally, for any positive integer $r$, the $r^{\rm th}$ 
{\em higher weight} $d_r = d_r(C)$ of the code $C$ is defined by
$$
d_r(C) := \min\left\{ \Vert D \Vert : D \mbox{ is a subspace of $C$ with } \dim D =r\right\}.
$$
Note that $d_1(C)=d(C)$. 
If $C$ is {\em nondegenerate}, that is, if
$C$ is not contained in a coordinate hyperplane of $\Fq^n$, then it is easy to see that
$$ 
0< d_1(C) < d_2(C) < \cdots < d_k(C) = n.
$$
See, for example, \cite{TV2} for a proof as well as a great deal of basic information about
higher weights of codes. The set $\{d_r(C):  1\le r \le k\}$ is often referred to as the
{\em weight hierarchy} of the code $C$. It is usually interesting, and difficult, to   
determine the weight hierarchy of a given code. Again, we refer to  \cite{TV2} for a variety
of examples, such as affine and projective Reed-Muller codes, codes associated to Hermitian
varieties or Del Pezzo surfaces, hyperelliptic curves, etc., where the weight hierarchy 
is completely or partially known. 

The following elementary result will be useful in the sequel. It appears, for example, in \cite[Lemma 2]{klove}. We include a proof for the sake of completeness. 

\begin{lemma}
\label{norm_formula}
Let $D$ be a $r$-dimensional code of a $[n,k]_q$-code $C$. Then 
$$
\Vert D \Vert = \frac{1}{q^r - q^{r-1}} \sum_{x\in D} \Vert x \Vert .
$$
In particular,
$$
d_r(C) = \frac{1}{q^r - q^{r-1}} \min\left\{\sum_{x\in D} \Vert x \Vert  :   D \mbox{ is a subspace of $C$ with } \dim D =r\right\}.
$$
\end{lemma}

\begin{proof}
Clearly, $(x,i)\mapsto (i,x)$ gives a bijection of $\{(x,i): x\in D \mbox{ and } i\in \supp (x) \}$ onto 
$\{(i,x): i\in \supp (D), \ x\in D \mbox{ and } x_i\ne 0 \}$. Hence 
$$
\sum_{x\in D} \Vert x \Vert = \sum_{x\in D} \, \sum_{i\in \supp (x)} 1 
= \sum_{i \in \supp(D)} \mathop{\sum_{x\in D}}_{x_i\ne 0} 1 
= \sum_{i \in \supp(D)} (q^r - q^{r-1}) = (q^r - q^{r-1}) \Vert D \Vert,
$$
where the penultimate equality follows by noting that if $i\in \supp (D)$, then $x\mapsto x_i$ defines a nonzero linear map of $D\to \Fq$.
\end{proof}

We remark that the Griesmer bound as well as the Griesmer-Wei bound is an easy consequence of the above lemma. In fact, as we shall see below, it can  also be used to derive a useful generalization of the Griesmer-Wei bound. To this end, we need to look at the elements of minimum Hamming weight as well as the second lowest positive exponent in the weight enumerator polynomial of $C$, provided of course this polynomial has at least two terms with 
positive exponents. 

Let $C$ be a linear $[n,k]_q$-code. 
Given any subspace $D$ of $C$, we let 
$$
\Delta (D):=\left\vert \left\{x\in D : \Vert x \Vert = d(C) \right\} \right\vert.
$$
Given any $r\in \Z$ with $1\le r\le k$, we let 
$$
\Delta_r(C):=\max\left\{\Delta (D) : D \mbox{ is a subspace of $C$ with } \dim D =r\right\}.
$$
Further, upon letting $S_C:= \left\{\Vert x\Vert : x\in C \mbox{ with } \Vert x \Vert > d(C) \right\}$, we define 
$$
e(C):=\left\{\begin{array}{ll}
\min S_C & \mbox{ if $S_C$ is nonempty,} \\
d(C) & \mbox{ if $S_C$ is the empty set.} \end{array} \right. 
$$
It may be noted that $e(C)\ge d(C)$ and also that the equality holds if and only if $\Delta_k(C)=q^k-1$. We are now ready to prove a simple, but useful  
generalization of the Griesmer-Wei bound. 

\begin{theorem}
\label{thm:eC}
Let $C$ be a linear $[n,k]_q$-code and $r$ be an integer 
with $1\le r\le k$. Then 
$$
d_r(C) \ge \frac{d(C) \Delta_r(C) +e(C)(q^r - 1 - \Delta_r(C))}{q^r- q^{r-1}}.
$$
\end{theorem}

\begin{proof}
Let $D_r$ be a $r$-dimensional subspace of $C$ such that 
$$
\sum_{x\in D_r} \Vert x \Vert = 
\min\left\{\sum_{x\in D} \Vert x \Vert  :   D \mbox{ is a subspace of $C$ with } \dim D =r\right\}.
$$
Then $D_r$ has $q^r-1$ nonzero elements and so, in view of Lemma \ref{norm_formula}, we have 
\begin{eqnarray*}
\left(q^r - q^{r-1}\right)d_r(C) & = & \mathop{\sum_{x\in D_r}}_{\Vert x \Vert = d(C)} d(C) + \mathop{\sum_{x\in D_r}}_{\Vert x \Vert > d(C)} \Vert x \Vert \\
& \ge  & d(C) \Delta(D_r) + e(C) \left(q^r-1-\Delta(D_r)\right) \\
& \ge  & e(C) \left(q^r-1 \right) -  \Delta_r(C) \left(e(C)-d(C)\right), 
\end{eqnarray*}
where the last inequality follows since $\Delta(D_r)\le \Delta_r(C)$ and $d(C)\le e(C)$.  
This yields the desired formula. 
\end{proof}

\begin{corollary}[Griesmer-Wei Bound]
Given any linear $[n,k]_q$-code $C$, we have
$$
d_r(C) \ge \sum_{j=0}^{r-1} \left\lceil \frac{d(C)}{q^j}\right \rceil  \qquad 
\mbox{ for } 1\le r\le k.
$$
\end{corollary}

\begin{proof} 
Using Theorem \ref{thm:eC} and the fact that $e(C) \ge d(C)$, we see that
$$
d_r(C) \ge \frac{d(C) (q^r - 1)}{q^r- q^{r-1}} = \sum_{i=0}^{r-1} \frac{d(C) q^i}{q^{r-1}} = \sum_{j=0}^{r-1} \frac{d(C)}{q^j} 
\ge \sum_{j=0}^{r-1} \left\lceil \frac{d(C)}{q^j} \right\rceil 
$$
for any integer $r$ with $1\le r\le k$. 
\end{proof}

\section{The Grassmann Code $C(\ell ,m)$}
\label{sec:clm}

Let us fix, throughout this section, a prime power $q$ and integers $\ell, m$ with $1\le \ell \le m$, and let 
$$
 n:= 
 {{m}\brack{\ell}}_q, \quad k := {{m}\choose{\ell}}, \quad \mbox{and} \quad \delta:=\ell(m-\ell),
$$
where ${{m}\brack{\ell}}_q$ is the Gaussian binomial coefficient, which was defined in Section \ref{sec:intro}. 
It may be remarked that ${{m}\brack{\ell}}_q$ is a polynomial in $q$ of degree $\delta$ with positive integral coefficients. 
The Grassmann code $C(\ell,m)$ is the linear $[n,k]_q$-code associated to the projective system corresponding to the Pl\"ucker embedding of the $\Fq$-rational points of the Grassmannian $G_{\ell,m}$ in $\PP^{k-1}_{\Fq} = \PP\big(\bigwedge^{\ell}\Fq^m\big)$; see, for example, \cite{N,GL} 
for greater details. Alternatively, $C(\ell,m)$ may be defined as follows. 

Let $V:=\Fq^m$. Fix a basis $\{e_1, \dots, e_m\}$ of $V$. Then we can, and will, fix a corresponding basis of $\bigwedge^{\ell} V$ given, in the notations of Section \ref{sec:clm}, by  $\{e_{\alpha}: \alpha\in I(\ell,m)\}$. 
Let $G_{\ell,m}= G_{\ell,m}(\Fq)$ be the Grassmann variety consisting of 
all $\ell$-dimensional subspaces of $V$. The Pl\"ucker embedding 
$G_{\ell,m}\hookrightarrow \PP\big(\bigwedge^{\ell}V \big)$ simply 
maps a $\ell$-dimensional subspace of $V$ spanned by 
$v_1, \dots , v_{\ell}$ to the point of 
$\PP\big(\bigwedge^{\ell} V \big)$ corresponding to 
\mbox{$v_1\wedge \cdots \wedge v_{\ell}$}. It is well-known that this embedding is well defined and nondegenerate. 
Fix representatives $\omega'_1, \dots , \omega'_n$ in $\bigwedge^{\ell} V$ corresponding to distinct points of $G_{\ell,m}(\Fq)$. We denote the subset $\{\omega'_1, \dots , \omega'_n\}$ of $\bigwedge^{\ell} V$ by $T(\ell,m)$. Having fixed a basis of $V$, we can identify each element of $\bigwedge^{m}V$ with a unique scalar in $\Fq$. With this in view, we obtain a linear map 
$$
\tau: {\textstyle{\bigwedge}}^{m-\ell}V\rightarrow \Fq^n \quad \mbox{ given by } \quad  \tau(\omega): = \left(\omega'_1\wedge\omega, \,  \omega'_2\wedge\omega, \, 
\dots, \, \omega'_n\wedge\omega\right).
$$
Since the Pl\"ucker embedding is nondegenerate, it follows that $\tau$ is injective. The Grassmann code $C(\ell, m)$ is defined as the image of the 
map $\tau$. It is clear that $C(\ell, m)$ is a linear $[n,k]_q$-code. 
Given any codeword $c\in C(\ell, m)$, there is unique $\omega \in \bigwedge^{m-\ell}V$ such that $\tau(\omega)=c$; we denote this $\omega$ by $\omega_c$. 

Given any  subspace $\E$ of $\bigwedge^{\ell} V$, we let $g(\E):= |\E\cap T(\ell,m)|$. Note that since $T(\ell,m)$ consists of nonzero elements, no two of which are proportional to each other, we always have 
\begin{equation}
\label{bdforgE}
\left|g(\E)\right|\le \frac{q^r-1}{q-1} \quad \mbox{ for any 
subspace $\E$ of $\bigwedge^{\ell} V$ with } \dim \E =r.
\end{equation}
Given any integer $s$ with $1\le s\le k$, we let
$$
g_s(\ell,m): = \max\left\{g(\E) : \E \mbox{ a subspace of $\bigwedge^{\ell} V$ of codimension } s\right\}.
$$
Note that as a consequence of \eqref{bdforgE}, we have
\begin{equation}
\label{bdforgslm}
g_s(\ell,m) \le \frac{q^r-1}{q-1} \quad \mbox{ where } \quad r:=k-s.
\end{equation}

\begin{lemma}
\label{NminusG}
Let $D$ be a 
subspace of $C(\ell, m)$ and $s=\dim D$. If $\D := \tau^{-1}(D)$, then  
$\E:=\D^{\perp}:=\{\omega'\in \bigwedge^{\ell} V : \omega'\wedge\omega=0 \}$ 
is a subspace of $\bigwedge^{\ell} V$ of codimension $s$ and 
$$\Vert D \Vert = n - g(\E).$$
\end{lemma}

\begin{proof}
Since $\tau$ is an isomorphism of $ \bigwedge^{m-\ell} V$ and $C(\ell,m)$, we have $\dim \D = s$. Also, since 
$(\omega', \omega) \mapsto \omega'\wedge\omega$ gives a nondegenerate bilinear map of $\bigwedge^{\ell} V \times \bigwedge^{m-\ell} V \to \Fq$, 
and  so $\E:=\D^{\perp}$ is a subspace $\bigwedge^{\ell} V$ 
of codimension $s$.  
For $1\le i\le n$, we have 
$$
i\not\in \supp (D) \Longleftrightarrow \omega'_i\wedge \omega = 0 \mbox{ for all } \omega \in \D \Longleftrightarrow \omega'_i\in \E.
$$
It follows that $\Vert D \Vert = n - g(\E).$
\end{proof}

\begin{corollary}
\label{dsNG}
$d_s\left(C(\ell, m)\right) = n - g_s(\ell, m)$ for  $s=1, \dots ,k$.
\end{corollary}

\begin{proof}
Clearly,  $\E\mapsto \tau(\E^{\perp})$ sets up a one-to-one correspondence between subspaces of $\bigwedge^{\ell} V$ of codimension $s$ and subspaces of $C(\ell,m)$ of dimension $s$. Hence the desired result follows from Lemma \ref{NminusG}.
\end{proof}

We now recall some important results of Nogin \cite{N}. Combining Theorem 4.1, Proposition 4.4 and Corollary 4.5 of \cite{N}, we have the following.

\begin{proposition}
\label{noginminwt}
The minimum distance of $C(\ell, m)$ is $q^{\delta}$ and the codewords $c$ 
of $C(\ell,m)$ such that $\omega_c$ is   decomposable attain the minimum weight $q^{\delta}$. 
Moreover, the number of minimum weight codewords in $C(\ell,m)$ is $(q-1)n$. 
\end{proposition}

A useful consequence is the following.

\begin{corollary}
\label{cornogin}
Given any $c\in C(\ell, m)$, we have
$$
\Vert c\Vert = q^{\delta} \Longleftrightarrow \omega_c \mbox{ is   decomposable.} 
$$
Moreover $\Delta(C(\ell,m)) = (q-1)n$.
\end{corollary}

\begin{proof}
The implication $\Longleftarrow$ follows from Proposition \ref{noginminwt}.
The other implication also follows 
from Proposition \ref{noginminwt} by noting that 
the number of   decomposable elements of $\bigwedge^{m-\ell}V$ is equal to
the number of   decomposable elements of $\bigwedge^{\ell}V$, 
and that the latter is equal to $(q-1)n$.
\end{proof}

In \cite{N}, Nogin goes on to determine some of the higher weights of $C(\ell, m)$ using Proposition \ref{noginminwt} and some additional work. More precisely, he proves formula \eqref{uptomu} in Section \ref{sec:intro}. As remarked in Section \ref{sec:intro},  Introduction, 
alternative proofs of \eqref{uptomu} are given in \cite{GL} as well as \cite{HJR}. The latter also proves the 
dual version \eqref{dualweights}. We give below yet another proof of 
\eqref{uptomu} and \eqref{dualweights} as an application 
of Theorem \ref{structure} and Corollary \ref{cornogin}. 

\begin{theorem}
\label{uptomuanddual}
Let $\mu:= \max \{\ell, m - \ell \} + 1$. Then for $0\le r \le \mu$ we have
$$
d_r\left( C(\ell ,m)\right) = q^{\d} + q^{\d - 1} +\dots +   q^{\d - r +1}
\quad {\rm and } \quad 
d_{k-r}\left( C(\ell ,m)\right) = n -( 1+ q + \cdots + q^{r -1}).
$$
\end{theorem}

\begin{proof}
The case $r=0$ is trivial. Assume that $1\le r\le \mu$. By Corollary \ref{cor:struturethm}, there is a decomposable subspace $E$ of $\bigwedge^{\ell}V$ of dimension $r$. Then $\hodge (E)$ is a   decomposable subspace of $\bigwedge^{m-\ell} V$ and hence by Corollary \ref{cornogin}, $D:=\tau(\hodge(E))$ is a $r$-dimensional subspace of $C(\ell,m)$ in which every nonzero vector is of minimal weight. Consequently, 
by Lemma \ref{norm_formula}, we have 
$$
\Vert D \Vert =  \frac{1}{q^r - q^{r-1}} \sum_{c\in D} \Vert c \Vert = 
 \frac{d\left( C(\ell ,m)\right) (q^r-1)}{q^r - q^{r-1}} = \sum_{j=0}^{r-1}  \frac{d\left( C(\ell ,m)\right) }{q^j} = \sum_{j=0}^{r-1}  q^{\d - j} .
$$
In other words, the Griesmer-Wei bound is attained. This proves the desired formula for $d_r\left( C(\ell ,m)\right)$.  
Next, $E$ is a subspace of $\bigwedge^{\ell}V$ of codimension $k-r$, and since $E$ is   decomposable, every $\omega'\in E$ with $\omega' \ne 0$ can be uniquely written as $\omega'=\lambda\omega'_i$ where $\lambda\in \Fq\setminus\{0\}$ and $i\in \{1,\dots , n\}$. It follows that $g(E) = (q^r-1)/(q-1) = 1+q + \cdots + q^{r-1}$, and so, in view of 
\eqref{bdforgslm}, we find $g_{k-r}(\ell,m) = 1+q + \cdots + q^{r-1}$. This,  together with Corollary \ref{dsNG}, yields the desired formula for $d_{k-r}\left( C(\ell ,m)\right)$.
\end{proof}

\section{Higher Weights of the Grassmann Code $C(2,m)$}
\label{sec:c2m}

The results  on the higher weights of $C(2,m)$ mentioned in the Introduction  will be proved in this section. 
Throughout, let $q,\ell,m,k,n,\delta$ be as in Section \ref{sec:clm}, except we set $\ell =2$. Also, we let $F:=\Fq$ and $V:=\Fq^m$. Note that the complete weight hierarchy of $C(2,m)$ is 
easily obtained from Theorem \ref{uptomuanddual} if $m\le 4$. With this in view, we shall assume that $m>4$. In particular, $\mu:=\max\{\ell,m-\ell\}+1 = m-1$.

We begin by recalling a result of Nogin  concerning the spectrum of 
$C(2,m)$. To this end, given any nonnegative integer $t$, let $N(m,2t)$ denote the number of skew-symmetric bilinear forms of rank 
$2t$ on $\Fq^m$. We know from \cite[\S 15.2]{MS} that 
\begin{equation}
\label{nm2t}
N(m,2t)
= \frac{(q^m-1)(q^{m-1}-1)\cdots(q^{m-2t+1}-1)}{(q^{2t}-1)((q^{2t-2}-1)\cdots (q^{2}-1)}q^{t(t-1)}.
\end{equation}
The said result of Nogin \cite[Thm. 5.1]{N} is the following. 

\begin{proposition}
\label{spectrumc2m}
Given any $i\ge 0$, let $A_i:=\left|\left\{c\in C(2,m) : \Vert c\Vert = i\right\}\right|$. 
Then 
\begin{equation}\label{spectrum}
A_i=
\left\{\begin{array}{ll}
N(m,2t) & \mbox{ if } i=\displaystyle q^{2(m-t-1)}\frac{q^{2t}-1}{q^2-1} \mbox{ for } 0\leq t \le \lfloor {m}/{2} \rfloor,\\
0 & \mbox{ otherwise.}
\end{array}\right.
\end{equation}
Moreover, for any $c\in C(2,m)$ and $0\leq t \le \lfloor {m}/{2} \rfloor$, we have
$$
\Vert c\Vert = \displaystyle q^{2(m-t-1)}\frac{q^{2t}-1}{q^2-1} \Longleftrightarrow \rank(\omega_c) = 2t.
$$
\end{proposition}

\begin{corollary} 
\label{cornoginde}
$d\left(C(2,m)\right) = q^{\delta}$ and 
$e\left(C(2,m)\right) = q^{\delta} +  q^{\delta -2}$.
\end{corollary}

\begin{proof}
The numbers $\theta_t:=  q^{2(m-t-1)}\frac{q^{2t}-1}{q^2-1}$ increase with $t$ and the first two positive values of $\theta_t$ ($t\ge 0$) are $q^{\delta}$ and $q^{\delta} +  q^{\delta -2}$.
\end{proof}

We now prove a number of auxiliary results needed to prove the main theorem. 

\begin{lemma}
\label{lemma_addonevec}
Let $E$ and $E_1$ be  subspaces of $\bigwedge^2V$ such that $E\subset E_1$ and $\dim E_1 = \dim E + 1$. Assume that $E$ is   decomposable and 
$E_1$ is not   decomposable. Then we have the following. 
\begin{enumerate}
	\item[{\rm (i)}] The set $E_1\setminus E$ contains at most $q^2(q-1)$   decomposable vectors. 
	\item[{\rm (ii)}] If $E_1\setminus E$ contains a   decomposable vector $\omega$ such that $V_{\omega}\subseteq V^E$, then $E_1\setminus E$ contains exactly $q^2(q-1)$   decomposable vectors.
\end{enumerate}
\end{lemma}

\begin{proof}
Both (i) and (ii) hold trivially if  $E_1 \setminus E$ contains no   decomposable vector. Now, suppose $E_1 \setminus E$ contains a   decomposable vector, say $\omega$. Then $E_1=E+F\omega$. 
Write $\omega=u\wedge v$, where $u,v\in V$, and let $r:=\dim E$. By 
Theorem \ref{structure}, we are in either of the following two cases. 
\smallskip

{\em Case 1:} $E$ is close of type I.
\smallskip

In this case, there are linearly independent elements $f,g_1, \dots, g_r\in V$ such that $E=\vspan\{f\wedge g_i : i=1, \dots, r\}$. Let $G:= \vspan\{g_1, \dots , g_r\}$. Elements $\xi$ of $E_1$ 
are of the form $\xi = f\wedge g + \lambda (u\wedge v)$, where $g\in G$ and $\lambda\in \Fq$. Clearly, $\xi$ and $(g, \lambda)$ determine each other uniquely, and $\xi \in E_1\setminus E$ if and only if $\lambda\ne 0$.
Observe that $\{f,u,v\}$ is linearly independent, lest 
we can write $u\wedge v = f\wedge h$ for some $h\in V$, and consequently, $E_1$ becomes   decomposable. 
Hence, by Corollary \ref{lemma_ab}, we see that if $\lambda\ne 0$, then 
$\xi= f\wedge g + \lambda (u\wedge v)$  is   decomposable if and only if $g\in \vspan\{f,u,v\}$. 
Further, in view of Lemmas \ref{elemlemma1} and \ref{lem:VsubEandsupE}, we have $V_{\omega}=\vspan\{u,v\}$ and $V^E= \vspan\{f,g_1, \dots , g_r\}$. Thus, $g\in \vspan\{f,u,v\}$  if and only if $f\wedge g = f\wedge x$  for some  $x\in V_{\omega}\cap V^E$. It follows that
  decomposable elements of $ E_1\setminus E$ are precisely of the form $f\wedge x + \lambda (u\wedge v)$, where $x\in V_{\omega}\cap V^E$ and $\lambda\in \Fq\setminus\{ 0\}$. Since $\left|V_{\omega}\cap V^E \right|\le \left|V_{\omega}\right| = q^2$ and $\left| \Fq\setminus\{ 0\}\right| = q-1$, both (i) and (ii) are proved. 
\smallskip

{\em Case 2:} $E$ is close of type II, but not closed  of type I. 
\smallskip

In this case, by Corollary \ref{cor:struturethm}, we must have $\dim E =3$. Thus, there are linearly independent elements $g_1, g_2,g_3\in V$ such that 
$E=\vspan\{g_2\wedge g_3, g_1\wedge g_3, g_1\wedge g_2\}$. Let 
$G:=\vspan\{g_1, g_2 , g_3\}$. Note that since $G=V^E$ and $\omega=u\wedge v\not\in E$, the possibility that $V_{\omega}\subseteq V^E$ does not arise in this case. Thus $\dim V_{\omega}\cap V^E \le 1$ and (ii) holds vacuously. 
The elements of $E_1$ are of the form $\xi = g\wedge h + \lambda (u\wedge v)$, where $g,h\in G$ and $\lambda\in \Fq$. Clearly, $\xi$ is a   decomposable element of $ E_1\setminus E$ if $g\wedge h =0$ and $\lambda \ne 0$. If, in addition, $\xi = g\wedge h + \lambda (u\wedge v)$ is   decomposable for some $g,h\in G$ with $g\wedge h \ne 0$ and $\lambda\in \Fq\setminus\{0\}$, then by Corollary \ref{lemma_ab}, 
$\{g,h,u,v\}$ is linearly dependent, and hence $\dim V_{\omega}\cap V^E = 1$. So we may assume without loss of generality that $u=g_1$. Then it is clear that the elements of $E_1\setminus E$ are precisely the (unique) linear combinations of the form $\lambda (u\wedge v) + \lambda_1 (g_2\wedge g_3)+ \lambda_2 (g_1\wedge g_3)+ \lambda_3 (g_1\wedge g_2)$, where $\lambda\in \Fq\setminus\{0\}$ and $\lambda_1, \lambda_2, \lambda_3\in \Fq$; moreover, by  Corollary  \ref{lemma_ab}, such a linear combination is   decomposable if and only if $\lambda_1=0$. It follows that $E_1\setminus E$ contains at most $q^2(q-1)$   decomposable elements.
\end{proof}

The bound $q^2(q-1)$ in Lemma \ref{lemma_addonevec} can be improved if 
the dimension of the   decomposable subspace $E$ is small. 

\begin{lemma}
\label{lemma_addonevectosmall}
Let $E$ and $E_1$ be  subspaces of $\bigwedge^2V$ such that $E\subset E_1$ and $\dim E_1 = \dim E + 1$. Assume that $E$ is   decomposable of dimension $r \ge 1$ and 
$E_1$ is not   decomposable. 
Then $E_1\setminus E$ contains at most $q^{r-1}(q-1)$   decomposable elements.
\end{lemma}

\begin{proof}
If $r\ge 3$, then the result 
is an immediate consequence of part (i) of Lemma \ref{lemma_addonevec}. 
Also, the result holds trivially if $E_1\setminus E$ contains no   decomposable element. Thus, let us assume that $r\le 2$ and $E_1=E+F\omega$, where $\omega\in E_1\setminus E$ is    decomposable. 

First, suppose $r=1$. Then $E=F\omega_0$ for some   decomposable 
$\omega_0\in \bigwedge^2V$. Since  $E_1$ is not   decomposable and $\omega\not\in E$, in view of Lemmas \ref{elemlemma1} and \ref{structure_lemma1}, 
we see that $\dim V_{\omega}\cap V_{\omega_0} = 0$. Hence 
from Corollary \ref{lemma_ab}, it follows that the only   decomposable elements in $E_1\setminus E$ are those of the form $\lambda \omega$ where $\lambda\in \Fq\setminus\{0\}$. Thus $E_1\setminus E$ contains at most $(q-1)$   decomposable elements, as desired. 

Next, suppose $r=2$. Then in view of Theorem \ref{structure}, there are linearly independent elements $f, g_1, g_2\in V$ such that $E = \vspan\{f\wedge g_1, f\wedge g_2\}$. As in the proof 
of Lemma \ref{lemma_addonevec}, we can write $\omega = u\wedge v$, where $u,v\in V$ are such that $\{f,u,v\}$ is linearly independent. Further, if $\dim V_{\omega}\cap V^E = 2$, then $V_{\omega}\subseteq V^E$ and we may assume without loss of generality that $g_1, g_2\in V_{\omega}$; hence 
$E_1 = \vspan\{f\wedge g_1, f\wedge g_2,g_1\wedge g_2\}$, and so $E_1$ is close of type II, which is a contradiction. Thus $\dim V_{\omega}\cap V^E < 2$ and so $\left|V_{\omega}\cap V^E \right|\le  q$. Moreover, as in the proof of Lemma \ref{lemma_addonevec}, 
  decomposable elements of $ E_1\setminus E$ are precisely of the form $f\wedge x + \lambda (u\wedge v)$, where $x\in V_{\omega}\cap V^E$ and $\lambda\in \Fq\setminus\{ 0\}$. Thus $E_1\setminus E$ contains at most $q(q-1)$   decomposable elements, as desired. 
\end{proof}

\begin{lemma}
\label{lem:maxcdvectors}
There exists a $(\mu+1)$-dimensional subspace of $\bigwedge^2V$ 
containing exactly $\left(q^{\mu}-1\right)+q^2\left(q-1\right)$ 
decomposable vectors.  Moreover, the remaining $\left(q^{\mu}-q^2\right)\left(q-1\right)$ nonzero elements in this subspace are of rank $4$.
\end{lemma}

\begin{proof}
By Corollary \ref{cor:struturethm}, there exists a $\mu$-dimensional   decomposable subspace of $\bigwedge^2V$, say $E$. 
Since $m>4$, we have $\mu>3$, and so by Theorem \ref{structure} and Corollary \ref{cor:struturethm}, 
$E$ is close of type~I. Thus there exist $\mu+1$ linearly independent elements $f, g_1, \dots, g_\mu \in V$ such that $E=\vspan\{f\wedge g_i:i=1,\dots, \mu\}$. Now, consider $\omega:= g_1\wedge g_2$ and $E_1 := E + F\omega$. It is clear that $\omega\not\in E$ and $E_1$ is not 
  decomposable. Moreover, by Theorem \ref{structure} and Lemma \ref{lem:VsubEandsupE}, $\dim V^E=\mu+1 = m$, and thus $V^E= V\supseteq V_{\omega}$. So it follows from part (ii) of Lemma \ref{lemma_addonevec} that $E_1\setminus E$ contains exactly $q^2(q-1)$   decomposable elements. Since every nonzero element of $E$ is   decomposable, we see that $E_1$ is a $(\mu+1)$-dimensional subspace of $\bigwedge^2V$  
containing exactly 
$\left(q^{\mu}-1\right)+q^2\left(q-1\right)$ 
  decomposable vectors. Since every element of $E_1$ is of the form $a (f\wedge g) + b(g_1\wedge g_2)$ for some $a,b\in F$ and $g\in \vspan\{g_1, \dots, g_\mu \}$, it follows from Corollary \ref{cor:sigmapirank2} and Corollary \ref{cor:sigmarank4} that the remaining $\left(q^{\mu+1}-1\right) - \left(q^{\mu}-1\right) - q^2\left(q-1\right)$ elements are of rank $4$. 
\end{proof}

\begin{lemma}
\label{lem:atmostcdvectors} 
Every $(\mu+1)$-dimensional subspace of $\bigwedge^2V$  contains 
at most 
$\left(q^{\mu}-1\right)+q^2\left(q-1\right)$ decomposable vectors.
\end{lemma}

\begin{proof} 
Let $E^*$ be any $(\mu+1)$-dimensional subspace of $\bigwedge^2V$.  
Let $r$ be the maximum among the dimensions of all   decomposable subspaces of $E^*$. If $r=0$, then $E^*$ contains no   decomposable element and the assertion holds trivially. Assume that $r\ge 1$. Let $E_r$ be a   decomposable $r$-dimensional subspace of $E^*$. Extending a basis of $E_r$ to $E^*$, we obtain a subspace  $E'$ of $E^*$ such that $E_r\cap E'=\{0\}$ and $E^*=E_r+E'$. 
Clearly, 
\begin{equation}
\label{unions}
E^* = \bigcup_{\omega\in E'} E_{r} +F\omega \quad \mbox{ and } \quad 
E^*\setminus E_{r} = \bigcup_{0\ne \omega\in E'} \left(E_{r} +F\omega\right)\setminus E_{r}.
\end{equation} 
Given any nonzero $\omega\in E'$, the space $E_{r} +F\omega$ is not   decomposable, thanks to the maximality of $r$, and so by part (i) of Lemma \ref{lemma_addonevec}, $\left(E_{r} +F\omega\right)\setminus E_{r}$ contains at most $q^2(q-1)$   decomposable elements. 
Moreover, for any nonzero $\omega, \omega'\in E'$, we have $E_{r} +F\omega = E_{r} +F\omega'$ 
if $\omega$ and $\omega'$ differ by a nonzero constant,
whereas $\left(E_{r} +F\omega\right) \cap \left( E_{r} +F\omega'\right) = E_{r}$ if $\omega$ and $\omega'$ do not differ by a nonzero constant. Thus the second decomposition in \eqref{unions} is disjoint if we let $\omega$ vary over nonzero elements of $E'$ that are not proportional to each other. It follows that  
$E^*\setminus E_{r}$ contains at most 
$\frac{q^2(q-1)|E'\setminus\{0\}|}{(q-1)} = q^2\left(q^{\mu+1-r}-1\right)$ 
decomposable elements. In case $r\le 2$, then using Lemma \ref{lemma_addonevectosmall} 
instead of part (i) of Lemma \ref{lemma_addonevec}, it follows that 
$E^*\setminus E_{r}$ contains at most $q^{r-1}\left(q^{\mu+1-r}-1\right)$   decomposable elements. Thus, if we let $s:=\min\{2, r-1\}$ and $N_r:= (q^r-1)+ q^s\left(q^{\mu+1-r}-1\right)$, then we see that $E^*$ contains at most 
$N_r$   decomposable elements. To complete the proof it suffices to observe that 
$$
\left(q^{\mu}-1\right)+q^2\left(q-1\right) - N_r  
= \left\{\begin{array}{ll} (q^r - q^3)(q^{\mu-r} - 1) &\mbox { if } r\ge 3, \\
(q^2 - q^{r-1})(q - 1) &\mbox { if } 1\le r\le 2,\end{array}\right.
$$
is always nonnegative. 
\end{proof}

\begin{corollary}
\label{cor:delta2mg2m}
We have $\Delta\left(C(2,m)\right) = \left(q^{\mu}-1\right)+q^2\left(q-1\right)$ and 
$g_{k-\mu-1}(2,m) = 1 + q + q^ 2 +\cdots + q^{\mu} + q^2$.
\end{corollary}

\begin{proof}
The assertion about $\Delta\left(C(2,m)\right)$ follows from Lemma \ref{lem:maxcdvectors}, 
Lemma \ref{lem:atmostcdvectors}, and Corollary \ref{cornogin}. Further, by Lemma \ref{lem:atmostcdvectors}, we see that if $\E$ is any  subspace of $\bigwedge^2V$ of codimension $k-\mu-1$, that is, of dimension $\mu + 1$, then 
$$
g(\E) \le \frac{ \left(q^{\mu}-1\right)+q^2\left(q-1\right)}{q-1} = 1 + q + q^ 2 +\cdots + q^{\mu-1} + q^2,
$$
and by Lemma \ref{lem:maxcdvectors}, we see that the 
bound is attained for some subspace of codimension $k-\mu-1$. This proves that $g_{k-\mu-1}(2,m) = 1 + q + q^ 2 +\cdots + q^{\mu-1} + q^2$. 
\end{proof}

\begin{theorem}
\label{mainthm}
For the Grassmann code $C(2,m)$, we have
$$
d_{k-\mu-1}\left(C(2,m)\right) = n-(1+q+\cdots+q^{\mu-1}+q^2).
$$
and
$$
d_{\mu + 1}\left(C(2,m)\right) =  q^{\d} + q^{\d - 1}+ 2q^{\d - 2}+ q^{\d - 3} +\dots +   q^{\d - \mu+1} 
$$
\end{theorem}

\begin{proof}
The formula for $d_{k-\mu-1}\left(C(2,m)\right)$ is an immediate consequence of Corollary \ref{cor:delta2mg2m} and Corollary \ref{dsNG}. To prove the formula for $d_{\mu + 1}\left(C(2,m)\right)$, we use Corollary \ref{cor:delta2mg2m} and Corollary \ref{cornoginde} to observe that for $C(2,m)$, the generalized Griesmer-Wei bound in 
Theorem \ref{thm:eC} can be written as 
$$
d_{\mu + 1}\left(C(2,m)\right) \ge q^{\d} + q^{\d - 1}+ 2q^{\d - 2}+ q^{\d - 3} +\dots +   q^{\d - \mu+1} .
$$
Moreover, by Lemma \ref{lem:maxcdvectors}, there exists a $(\mu+1)$-dimensional subspace, say $E_1$, of $\bigwedge^2V$ containing $ \left(q^{\mu}-1\right)+q^2\left(q-1\right)$ decomposable elements such that the remaining $\left(q^{\mu}-q^2\right)\left(q-1\right)$ nonzero elements are of rank $4$.  Thus, in view of 
Proposition \ref{spectrumc2m}, we see that 
$D_1:=\tau(\hodge(E_1))$ is a $(\mu+1)$-dimensional subspace of $C(2,m)$ in which 
$ \left(q^{\mu}-1\right)+q^2\left(q-1\right)$ elements are of weight $q^{\delta}$ while the remaining $\left(q^{\mu}-q^2\right)\left(q-1\right)$ nonzero elements are of weight $q^{\delta} + q^{\delta - 2}$. Consequently, by Lemma \ref{norm_formula}, we have 
\begin{eqnarray*}
\Vert D_1 \Vert &= & \frac{1}{q^{\mu+1} - q^{\mu}} \sum_{c\in D} \Vert c \Vert \\
& = & 
 \frac{q^{\delta} \left[\left(q^{\mu}-1\right)+q^2\left(q-1\right)\right]}{q^{\mu+1} - q^{\mu}}
  + \frac{\left(q^{\delta}+ q^{\delta - 2}\right)\left[\left(q^{\mu}-q^2\right)\left(q-1\right)\right]}{q^{\mu+1} - q^{\mu}}  \\ 
& = &   
q^{\delta -\mu}\left(q^{\mu} + q^{\mu-1}+ \cdots + q + 1\right) + q^{\delta - 2} - q^{\delta -\mu}.
\end{eqnarray*}
This proves that $d_{\mu + 1}\left(C(2,m)\right) = q^{\d} + q^{\d - 1}+ 2q^{\d - 2}+ q^{\d - 3} +\dots +   q^{\d - \mu+1}$.
\end{proof}

\begin{remark}
It appears quite plausible that the new pattern which emerges with $d_{\mu+1}\left(C(2,m)\right)$ continues for the next several values of $d_r\left(C(2,m)\right)$. More precisely, we conjecture that for 
$\mu + 1 < r \le 2\mu - 3$, one has
$$
d_r\left(C(2,m)\right) = \left(q^{\d} + q^{\d - 1}+ \cdots +   q^{\d - \mu+1}\right)
+ \left(q^{\delta-2}+q^{\delta-3}+\cdots+q^{\delta-r+\mu-1}\right)
$$
and 
\begin{equation*}
d_{k-r}\left(C(2,m)\right) =
 n- \left(1+q+\cdots+q^{\mu-1}\right) - \left(q^2+q^3+\cdots+q^{r-\mu+1}\right)
\end{equation*}
These conjectural formulae yield the complete weight hierarchy of $C(2,6)$. In general, we believe that the conjecture of Hansen, Johnsen and Ranestad \cite{HJR} about 
$d_r\left(C(\ell,m)\right) - d_{r-1}\left(C(\ell,m)\right)$ being a power of $q$ is likely 
to be true and that determining $d_r\left(C(\ell,m)\right)$ from $d_{r-1}\left(C(\ell,m)\right)$ is a matter of deciphering the correct term of the Gaussian binomial coefficient (which can be written as a finite sum of powers of $q$) 
that gets added to  $d_{r-1}\left(C(\ell,m)\right)$. 
\end{remark}

\end{document}